\documentclass[9.0pt, a4paper]{article}
\usepackage{amsmath}
\usepackage{fancyhdr}
\usepackage{color}
\usepackage{amsmath}
\usepackage{amsfonts}
\usepackage{dsfont}
\usepackage{mathrsfs}
\usepackage{latexsym}
\usepackage{graphicx}
\usepackage{epstopdf}
\usepackage{amssymb}
\usepackage{indentfirst}
\usepackage{subfigure}
\usepackage{booktabs}
\usepackage{authblk}
\usepackage{cite}
\usepackage{float}
\usepackage[numbers,sort&compress]{natbib}
\usepackage{caption}
\usepackage{booktabs}
\usepackage{threeparttable}
\usepackage{multicol}
\usepackage{multirow}
\usepackage{geometry}
\usepackage{lineno}
\usepackage{amsbsy}
\usepackage{bm}
\usepackage{dcolumn}
\usepackage{mathrsfs}
\usepackage[title]{appendix}
\usepackage{stmaryrd}
\usepackage{tikz}
\tikzset{>=stealth}
\newcommand{\norm}[1]{\lVert #1 \rVert}
\newtheorem{remark}{Remark}[section]
\newtheorem{theorem}{Theorem}[section]
\newtheorem{lemma}[theorem]{Lemma}
\newtheorem{corollary}[theorem]{Corollary}

\numberwithin{equation}{section}

\begin{document}
\captionsetup[figure]{labelfont={bf},labelformat={default},labelsep=period,name={Fig.}}
%\pagewiselinenumbers

\title{Pointwise A Posteriori Error Estimators for Multiple and Clustered Eigenvalue Computations}\date{}
\author{Zhenglei Li}
\author{Qigang Liang}
\author[1,*]{Xuejun Xu}
\affil[1]{\small School of Mathematical Science, Tongji University, Shanghai 200092, China and Key Laboratory of Intelligent Computing and Applications (Tongji University), Ministry of Education (2211186@tongji.edu.cn, qigang$\_$liang@tongji.edu.cn, xuxj@tongji.edu.cn)}

\maketitle

{\bf{Abstract}:}\ \ In this work, we propose an a pointwise a posteriori error estimator for conforming finite element approximations of eigenfunctions corresponding to multiple and clustered eigenvalues of elliptic operators. It is proven that the pointwise a posteriori error estimator is reliable and efficient, up to some logarithmic factors of the mesh size. The constants involved in the reliability and efficiency are independent of the gaps among the targeted eigenvalues, the mesh size and the number of mesh level. Specially, we obtain a by-product that edge residuals dominate the a posteriori error in the sense of $L^{\infty}$-norm when the linear element is used. With the aid of the weighted Sobolev stability of the $L^2$-projection, we also propose a new method to prove the reliability of the a posteriori error estimator for higher order finite elements. A key ingredient in the a posteriori error analysis is some new estimates for regularized derivative Green's functions. Some numerical experiments verify our theoretical results.

{\bf{Keywords}:}\ \ Elliptic eigenvalue problems, multiple and clustered eigenvalues, maximum norm, finite element methods, a posteriori error estimator \hspace*{2pt}

\section{Introduction} 
This paper considers an a posteriori error estimator for conforming finite element approximations of eigenfunctions corresponding to multiple and clustered eigenvalues of elliptic operators with adaptive finite element methods (AFEMs) in the sense of $L^{\infty}$-norm.   When solving elliptic partial differential equations, standard finite element methods on quasi-uniform meshes often exhibit suboptimal performance due to local singularities in the exact solution, which can significantly degrade numerical accuracy. To address this limitation, adaptive finite element methods (AFEMs) were introduced in \cite{MR483395}. Over the past several decades, AFEMs have been the subject of extensive theoretical research and have been successfully implemented in various practical applications, as documented in \cite{MR745088,AINSWORTH19971,MR2194587,i28} and the references therein. In the context of elliptic eigenvalue problems, a posteriori error estimators have been rigorously analyzed in \cite{i34,i32,i35,MR2476558}, where both the reliability and efficiency of these estimators are established. The convergence of AFEMs and the quasi-optimal convergence rate are obtained in \cite{i31,MR2430983,MR2747809}, the proofs in which are restricted to simple eigenvalues. It is noteworthy that the authors in \cite{MR3407250} first derive the a posteriori error estimator for multiple eigenvalues and prove the convergence of AFEMs. Subsequently, the results are extended to clustered eigenvalues in \cite{i35,MR3347459}.

\par  The a posteriori error estimates discussed above are measured with respect to the energy norm, ensuring convergence of the discrete solution to the exact solution in an integral-averaged sense. However, it should be noted that in many practical applications, pointwise error estimate is of considerable importance (see, for example, \cite{MR1839794, MR3549870, MR1079028} and the references therein). Consequently, the development and rigorous analysis of pointwise a posteriori error estimators in the $L^{\infty}$-norm attract  significant interest, which often poses substantial theoretical challenges.  In recent years, significant progress has been made in the a priori error analysis in the $L^{\infty}$-norm for finite element approximations of elliptic source problems. A comprehensive overview of these developments can be found in \cite{i2,i3,i4,i7,i13,MR4435940} and references therein. Regarding the a posteriori error estimation, the author in \cite{1} introduces pointwise a posteriori error estimators over general polygonal domains in $\mathbb{R}^{2}$ and establishes their equivalence to the a priori error, measured by the $L^{\infty}$-norm. Subsequently, these findings have been extended to three-dimensional settings and non-conforming finite element methods in \cite{2}. A crucial step in the proof of the reliability presented in \cite{1} involves deriving an a priori $W^{2,1}$-estimate for regularized Green’s functions. In the context of pointwise a priori error analysis for eigenvalue problems, the study \cite{i36} achieves a sharp convergence rate in convex domains, under the assumption that $u\in W^{2,\infty}(\Omega)$. For pointwise a posteriori error of symmetric elliptic eigenvalue problems, we recently (see \cite{simple}) design and analyze a pointwise a posteriori error estimator for eigenfunctions corresponding to simple eigenvalues. In practical computations, the pointwise a posteriori error estimator for eigenfunctions corresponding to simple eigenvalues can not capture well the singularities of eigenfunctions corresponding to multiple and clustered eigenvalues (see subsection 5.3). Therefore, it is significant to design and analyze the a posterior error estimator for eigenfunctions corresponding to multiple and clustered eigenvalues.

\par  In this work, we introduce a pointwise a posteriori error estimator and establish its reliability and efficiency with respect to the $L^{\infty}$-norm for eigenfunctions corresponding to multiple and clustered eigenvalues of symmetric elliptic operators. It is noted that analyzing the reliability and efficiency of the pointwise a posteriori error estimator is nontrivial. One first requires to analyze the relationship between constants (involved in the efficiency and reliability) and the gaps among targeted eigenvalues. Hence, it is necessary to estimate the angle between the exact invariant subspace and the discrete invariant subspace in the sense of $L^{\infty}$-norm, which poses more difficulties. The second difficulty arises from the distinction between eigenvalue problems and source problems, the absence of Galerkin orthogonality in the finite element discretization of eigenvalue problems, leads to a $L^{2}$-term, which must be carefully bounded. Third, the extension to higher-order methods needs a more careful estimate for element residuals in the a posteriori error estimator for multiple and clustered eigenvalues.

\par  To overcome these difficulties and obtain the reliability and efficiency of the a posteriori error estimator for eigenfunctions corresponding to multiple and clustered eigenvalues in the sense of $L^{\infty}$-norm, we take advantage of some mathematical tools and skills. A key ingredient is to derive some a priori estimates for the Ritz projection of regularized derivative Green's functions.  Moreover, with the help of the weighted Sobolev stability of the $L^2$-projection, we provide a new method to extend the a posteriori error  analysis to higher-order finite elements, which is different from \cite{MR3532806}. Specifically, we obtain the following result:
\begin{equation*}
          ||e_j||_{L^{\infty}(\Omega)}\lesssim \frac{|\ln{\underline{h}_l}|^2}{|\ln{\overline{h}_l}|} \frac{\lambda_{n+L}}{\lambda_{n+1}}\eta_l
\end{equation*}
for any $j\in J$ and
\begin{equation*}
        \eta_l \lesssim \sqrt{L}\sum_{j\in J}||e_j||_{L^{\infty}(\Omega)},
\end{equation*}
where $e_{j}=u_{j}-\Lambda_{l}u_{j}$ with $u_{j}$ being the $j$th eigenfunction and $\Lambda_{l}$ being defined in \eqref{deflambdal}, $\eta_{l}$ is the pointwise a posterior error estimator defined in \eqref{def eta}, $\lambda_{j}$ is the $j$th eigenvalue of the eigenvalue problem, $J=\{n+1,n+2,...,n+L\}$ and $\overline{h}_{l}$ ($\underline{h}_l$) is the maximum (minimum) of the mesh-size. Throughout this paper, $a\lesssim(\gtrsim)\ b$ means $a\leq(\geq)\ Cb$, where the letter $C$ is generic positive constant independent of the mesh size and the gaps among eigenvalues in a cluster, possibly different at each occurrence. Specially, as a by-product, we prove that edge residuals dominate the a posterori error in the sense of $L^{\infty}$-norm when the linear element is used.

\par This paper is organized as follows. In section 2, we introduce the model problem and some mathematical notations. Then we review some useful lemmas in \cite{MR3347459} and give a lower bound for the a priori error. In section 3, we derive some new a priori estimates for  regularized Green's functions and derivative Green's functions. In section 4, we  give the main result of this paper, namely, the reliability and efficiency of the pointwise a posteriori error estimator.  Some numerical experiments are given to verify our theoretical findings in section 5 and the conclusion is presented in section 6.

\section{Preliminaries}

\subsection{Triangulations and finite element spaces}
Let $\Omega\subset \mathbb{R}^2$ be a polygonal domain without restrictions on the size of the interior angels. Let $r_0=\frac{\pi}{\theta_0}$, where $\theta_0$ is the largest angle among the interior corners of $\Omega$. Throughout this paper, for any subset $\widetilde{\Omega} \subset \Omega$, we use the standard notations for the Sobolev spaces $W^{m,p}(\widetilde{\Omega})$ ($H^m(\widetilde{\Omega})$) and $W_{0}^{m,p}(\widetilde{\Omega})$($H_0^m(\widetilde{\Omega})$) (see \cite{MR450957}). Moreover, we denote by $L^p(\widetilde{\Omega})$ and $L^p(\Gamma)$ the usual Lebesgue spaces equipped with the standard norms $\Vert \cdot\Vert_{L^p(\widetilde{\Omega})}$ and  $\Vert \cdot\Vert_{L^p(\Gamma)}$, where  $\Gamma \subset \overline{\Omega}$ is a Lipschitz curve. 
\par We use $\{ {\mathcal{T}_l}\}_{l\in\mathbb{N}}$ to denote  a shape-regular family of nested conforming meshes over $\Omega$, i.e., there exists a constant $\gamma$ such that 
\begin{equation*}
    \frac{h_l|_T}{\rho_T}\leq \gamma,
\end{equation*}
where the piecewise constant mesh-size function $h_l := h_{\mathcal{T}_l}$ is defined by $h_l|_T := h_T := \text{meas}(T)^{\frac{1}{2}}$ for any element $T \in \mathcal{T}_l$ and $\rho_T$ is the diameter of the biggest ball contained in $T$. For simplicity, we assume that each element $T\in \mathcal{T}_l$ is a closed set. 
We set $ \overline{h}_l:=\max_{T \in \mathcal{T}_l} h_T$ and $ \underline{h}_l:=\min_{T \in \mathcal{T}_l} h_T$.

\par Now we introduce the jump operator. Given an interior edge $E$ shared by two elements $T_{+}$ and $T_{-}$, i.e., $E=\partial T_+ \cap \partial T_-$,
 for any piecewise function $v$, we define  
 \begin{equation*}
     [\![v]\!]_E=(v|_{K_+})|_E-(v|_{K_-})|_E
 \end{equation*}
 and
 \begin{equation*}
     [\![\nabla v]\!]_E=(\nabla v|_{K_+})|_E-(\nabla v|_{K_-})|_E.
 \end{equation*}
 By convention, if $E$ is a boundary edge, we set
\begin{equation*}
    [\![v]\!]_E=v|_{E}, \ \ \ \ [\![\nabla v]\!]_E=\nabla v |_{E}.
\end{equation*}
For simplicity, we denote by 
 \begin{align*}
    [\![\frac{\partial v}{\partial {n} 
 }]\!]_E &=[\![\nabla v]\!]_E \cdot {n}_E,
 \end{align*}
where ${n}_E$ is the unit outer normal vector of $E$. 
The choice of the particular normal is arbitrary, but is considered to be fixed once and for all. In subsequent sections, we will drop the subscript $E$ without causing confusion. 

\par We denote by $P_k(\mathcal{T}_l)$ the set of piecewise polynomial functions of degree $\leq k$ with respect to $\mathcal{T}_l$. The conforming finite element space $V_l^k$ is defined as
\begin{equation*}
    V_l^{k}=P_k(\mathcal{T}_l)\cap H_0^1(\Omega).
\end{equation*}

\subsection{Eigenvalue problem and its discretization}
We consider the model problem
\begin{align}\label{Laplace}
	\begin{cases}
		-\Delta u=\lambda u \ \ &\text{in} \ \Omega,\\
		u=0 \ \ &\text{on} \  \partial \Omega.
	\end{cases}
\end{align}
For notational convenience, we denote by 
\begin{equation*}
	(\phi ,\psi)_{\widetilde{\Omega}}:=\int_{\widetilde{\Omega}} \phi \psi\ \text{d}x,\ \ \ \ \langle \phi ,\psi \rangle_\Gamma:=\int_\Gamma\phi \psi\ \text{d}s.
\end{equation*}
If $\widetilde{\Omega}=\Omega$, we drop the subscript $\widetilde{\Omega}$ in $(\phi ,\psi)_{\widetilde{\Omega}}$. The variational form of \eqref{Laplace} is to find $(\lambda,u)\in \mathbb{R} \times V (:=H_{0}^{1}(\Omega)) $ such that
\begin{equation}\label{Laplace_variation}
  a(u,v)=\lambda(u,v)\ \ \ \ \forall\ v \in V,\\
\end{equation}
where 
\begin{equation*}
	a(v ,w)=\int_\Omega \nabla v \cdot \nabla w \ dx \ \ \ \ \forall\ v ,w \in V.
\end{equation*}
We use the $k$th order  Lagrangian finite element to discrete problem \eqref{Laplace}. The finite element discretization is to seek  discrete eigenpair $(\lambda_l,u_l)\in \mathbb{R}\times V_l(:=V_l^k)$ satisfies
\begin{equation}\label{def_u_h}
a(u_l,v_l)=\lambda_l(u_l,v_l) \ \ \ \ \forall \ v_l\in V_l.
\end{equation}
It is known that \eqref{Laplace} has countably many positive eigenvalues, with the only accumulation point $+\infty$. Suppose that the eigenvalues of \eqref{Laplace} and the discrete eigenvalues of \eqref{def_u_h} are enumerated  as 
\begin{equation*}
    0 < \lambda_1 \leq \lambda_2 \leq \dots \quad \text{and} \quad 0 < \lambda_{l,1} \leq \dots \leq \lambda_{l,\dim(V_{l})}.
\end{equation*}
Let $\{u_1, u_2, u_3, \dots\}$ $(\{u_{l,1}, u_{l,2}, \dots, u_{l,{\dim({V_{l}})}}\})$ denote some $L^2$-orthonormal bases of corresponding (discrete) eigenfunctions. For a cluster of eigenvalues $\lambda_{n+1}, \dots, \lambda_{n+L}$ of length $L \in \mathbb{N}$, we define the index set 
$$
J := \{ n+1, \dots, n+L \}
$$
and the spaces 
$$
W := \operatorname{span}\{ u_j \}_{j \in J} \quad \text{and} \quad W_{l} := \operatorname{span}\{ u_{l,j} \}_{j \in J}.
$$
We assume that  there holds a continuous gap condition, i.e., 
\begin{equation*}
    \lambda_n<\lambda_{n+1}\leq \lambda_{n+2}\leq ...\leq \lambda_{n+L}<\lambda_{n+L+1}
\end{equation*}
(with the convention $\lambda_0:= 0$).

\subsection{Some basic estimates for eigenvalue clusters}
We first introduce some operators. Define the Ritz  projection $R_l$: $V\rightarrow V_l$ such that
 \begin{equation}\label{def_R_l}
     a(R_lv,w_l)=a(v,w_l)\ \ \ \ \forall \ w_l\in V_l.
 \end{equation}
We use  $Q_l:L^2(\Omega)\rightarrow V_l$ to denote the $L^2$-projection defined by
\begin{equation*}
    (Q_lv,w_l)=(v,w_l) \ \ \ \ \forall\  w_l\in V_l
\end{equation*}
and $\Pi_l^k :C^0(\overline{\Omega})\rightarrow V_l^k$ to denote the usual Lagrange interpolation operator. Moreover, we denote by $P_l$ the $L^2$-projection to $W_l$. We also denote by $\Lambda_l$ the composite of $P_l$ and $R_l$, namely 
\begin{equation}\label{deflambdal}
    \Lambda_l =P_l\circ R_l.
\end{equation}

For any index $j\in J$, the function $\Lambda_l u_j\in W_l$ is regarded as the approximation of $u_j$ (in subsection 4.2, we explain that it is a reasonable approximation). Note that $\Lambda_l u_j$ is not necessarily an eigenfunction. 
 \par Now we shall review some useful results of eigenvalues clusters in \cite{MR3259027} and \cite{MR3347459}. Here we omit the details.
\begin{lemma}\label{lemma u-Lambda_l u}
Assume that there exists a separation bound
\begin{equation}\label{eq:H2}
    M_J=\sup_{l\in \mathbb{N}}\max_{i\in\{ 1,2,3,...,\text{dim}(V_l)\backslash J \}}\max_{j\in J}\frac{\lambda_j}{|\lambda_{l,i}-\lambda_j|}<\infty.
\end{equation}
Any eigenfunction $u_j\in W$ $(j\in J)$ with $\norm{u_j}_{L^2(\Omega)}=1$ satisfies 
    \begin{equation}\label{u-Lambda_l u}
        \norm{u_j-\Lambda_l u_j}_{L^2(\Omega)}\leq (1+M_J) \norm{u_j-R_lu_j}_{L^2(\Omega)}.
    \end{equation}
\end{lemma}

 \begin{lemma}\label{lemma Lambda_l}
     Any eigenpair $(\lambda_j,u_j)\in \mathbb{R}^+\times W$ $(j\in J)$ of \eqref{Laplace} satisfies 
     \begin{equation}\label{Lambda_l_v_l}
      a(\Lambda_lu_j,v_l)=\lambda_j(P_lu_j,v_l), \ \ \ \ \forall v_l\in V_l.
     \end{equation}
 \end{lemma}
Besides the above two lemmas, we shall derive an a priori lower bound for $\norm{u_j-\Lambda_lu_j}_{L^{\infty}(\Omega)} (j\in J)$, which will be used in the subsequent a posteriori error analysis.

\begin{lemma}[Lower bound]\label{lemma lower bound}
Let $(\lambda_j,u_j) \in \mathbb{R}^+\times W $ $(j\in J)$ be an eigenpair of \eqref{Laplace}. If the mesh-size $\overline{h}_l$ is small enough, we have the lower bound
 \begin{equation}\label{lower bound}
        \inf_{v_l\in V_l (\mathcal{T}_l)}\norm{u_j-v_l}_{L^{\infty}(\Omega)}\gtrsim \underline{h}_l^{2[\frac{k+1}{2}]},
    \end{equation}
 \end{lemma}
 where $[c]$ represents the rounding of the number $c$. 
\begin{proof}
 Let $z\in T_z$ ($T_z\in \mathcal{T}_l$) satisfy $|u_j(z)|=\norm{u_j}_{L^\infty(\Omega)}$.
 It is known that the following estimate holds for $1<p<\frac{4}{3}$ (see \cite{6,5}): 
\begin{equation}\label{priori}
    \vert v \vert_{W^{2,p}(\Omega)}\leq C_p \Vert \Delta v \Vert_{L^p(\Omega)}  \ \ \ \ \forall \ v \in W^{2,p}(\Omega)\cap W_0^{1,p}(\Omega).
\end{equation}
On using the Sobolev embedding theorem, in conjunction  with \eqref{priori}, we obtain
\begin{equation*}
    \norm{u_j}_{C^{0,\alpha_0}(\overline{\Omega})}\lesssim ||\lambda_j u_j||_{L^p(\Omega)}\lesssim 1
\end{equation*}
for some $\alpha_0>0$ and $1<p<\frac{4}{3}$.
Thus we may infer that
\begin{equation*}
    d:=\text{dist}(z,\partial\Omega)\gtrsim \left(\frac{\norm{u_j}_{L^{\infty}(\Omega)}}{\lambda_j}\right)^{\frac{1}{\alpha_0}}\gtrsim 1.
\end{equation*}
Let $k_0:= [\frac{k+1}{2}]$. We know that $u$ is smooth enough in the interior of $\Omega$. Using the standard interpolation theory and inverse estimates, for any $v\in P_{2k_0-1} (\mathcal{T}_l)$ we may deduce
   \begin{align*}
     {\lambda_j^{k_0}}\norm{u_j}_{L^{\infty}(\Omega)}&=  {|\Delta^{k_0}u_j(z)|}\lesssim|u_j|_{W^{2{k_0},\infty}({T_z})}\\
      &=|u_j-v|_{W^{2{k_0},\infty}({T_z})}\\
      &\leq |u_j-\Pi_l^{2{k_0}}u_j|_{W^{2{k_0},\infty}({T_z})}+|\Pi_l^{2{k_0}}u_j-v|_{W^{2{k_0},\infty}({T_z})}\\
      &\lesssim h_{T_z}|u_j|_{W^{2{k_0}+1,\infty}({T_z})}+h_{T_z}^{-2{k_0}}\norm{\Pi_l^{2{k_0}}u_j-v}_{L^{\infty}({T_z})}\\
      &\leq  h_{T_z}|u_j|_{W^{2{k_0}+1,\infty}({T_z})} +h_{T_z}^{-2{k_0}}(\norm{u_j-\Pi_l^{2{k_0}}u_j}_{L^{\infty}({T_z})}+\norm{u_j-v}_{L^{\infty}({T_z})})\\
      &\lesssim h_{T_z}|u_j|_{W^{2{k_0}+1,\infty}({T_z})}+h_{T_z}^{-2{k_0}}\norm{u_j-v}_{L^{\infty}({T_z})}.
    \end{align*}
It follows that 
\begin{equation*}
     \inf_{v_l\in P_{2k_0-1} (\mathcal{T}_l)} \norm{u_j-v}_{L^{\infty}({T_z})}\gtrsim h_{T_z}^{2k_0}-h_{T_z}^{2k_0+1}|u_j|_{W^{2{k_0}+1,\infty}({T_z})}
\end{equation*}
To drop the second term, we use the well-known interior estimate (see Theorem 5 in subsection 6.3.1 in \cite{MR2597943}) 
\begin{equation*}
    |u_j|_{H^{m}(U)}\lesssim 1
\end{equation*}
for any positive integer $m$ and $U\subset\subset \Omega$, where we use $U\subset\subset \Omega 
$ to denote $U$ is a subset of $\Omega $ satisfying $\text{dist}(\partial U, \partial\Omega)\gtrsim 1$. By the Sobolev embedding theorem and noting $d\gtrsim1$, we have
\begin{equation*}
    |u_j|_{W^{2{k_0}+1,\infty}({T_z})}\lesssim 1,
\end{equation*}
which leads to \eqref{lower bound} with a small enough $\overline{h}_l$.
\end{proof}

\section{Some a priori estimates}
In this subsection, we  introduce a  regularized (derivative) Green's function, and derive some useful a priori estimates, which will play an important role in the next section.  
\subsection{$W^{2,1}$-estimate for regularized Green's function}
\par Given a fixed point $x_0\in \overline{\Omega}$, define a smooth function $\delta_0$ satisfying (see \cite{1}):
\begin{equation}\label{prodelta1}
    \text{supp}\  \delta_0 \subset B(x_0,\rho),\ \ \int_{\Omega} \delta_0 \ \text{d} x=1,
\end{equation}
and
\begin{equation}\label{prodelta2}
     0\leq \delta_0 \lesssim \rho^{-2},\ \ \Vert  \delta_0 \Vert_{W^{s,\infty}(\Omega)}\lesssim \rho^{-s-2}, \ s=1,2,....
\end{equation}
where $\rho =h_{T}^\beta$ with $\beta\geq 1$ being an undetermined parameter and $T$ is an element containing $x_{0}$. We define the regularized Green's function $g_0$ such that
\begin{align}\label{defg}
    \begin{cases}
		-\Delta g_0=\delta_0\ \ \ \ &\text{in} \ \Omega,\\
		g_0=0\ \ \ \ &\text{on} \ \partial\Omega.       
    \end{cases}\ \  
\end{align}
For the following analysis, we review an $W^{2,1}$-estimate for $g_0$ for arbitrary polygonal domain (see \cite{simple})
\begin{equation}\label{smooth2}
      \vert g_0 \vert_{W^{2,1}(\Omega)}\lesssim \vert \ln\rho\vert.
\end{equation}

\subsection{Estimates for regularized derivative Green’s function}
\par For any element $T\in \mathcal{T}_l$, let $x_T\in T$. There exists $\delta_{1}\in C_0^{\infty}(T)$ (see \cite{construction2,construction1}) satisfying
\begin{equation}\label{prodT}
	 (\delta_1, v)_T=v(x_T) \ \ \forall\  v\in P_k(T),\  \Vert \delta_{1}\Vert_{W^{s,\infty}(\Omega)}\lesssim h_T^{-s-2},s=0,1,2,....
\end{equation} 
Define a regularized derivative Green's function $g_1$ scuh that
\begin{align}\label{def_g1}
	\begin{cases}
		-\Delta g_{1}=\nabla\delta_{1}\cdot {\nu},\ \  &\text{in} \ \Omega,\\
		g_{1}=0,\ \ &\text{on} \ \partial \Omega,
	\end{cases}
\end{align}
where ${\nu}=(\nu_1,\nu_2)$ is an arbitrary unit vector. We introduce the regularized derivative Green's function, which plays a crucial role in the subsequent a posteriori error analysis. We review a result about the gradient estimate in \cite{MR2141694}, which is beneficial to estimate the $W^{1,p}$-norm of $g_1$($g_0$).
\begin{lemma}\label{lemma gradient estimate}
    Given a Lipschitz domain $D$, let $U$ solve
    \begin{equation}\notag
    \begin{cases}
    -\Delta U=\text{div}F\ \ \ \ &x\in  D,\\
    U=0\ \ \ \ \ &x\in \partial D,
    \end{cases}
    \end{equation}
where $F\in (L^p(D))^{2}$. We have the following $W^{1,p}$-estimate 
\begin{equation}\label{gradient}
    \norm{U}_{W^{1,p}(D)}\lesssim \norm{F}_{L^{p}(D)},
\end{equation}
where $\frac{4}{3}- \varepsilon<p< 4+\varepsilon$ with $\varepsilon>0$ depending only on $D$.
\end{lemma}

Using Lemma \ref{lemma gradient estimate}, we may give  $W^{1,p}$-estimates for $g_0$ and $g_1$.
\begin{lemma}\label{estg_1}
      Let $g_0$ and $g_1$ be defined as \eqref{defg} and \eqref{def_g1}. The following estimates hold     \begin{equation}\label{estimate_g0_1p}
      \vert g_0\vert_{W^{1,p}(\Omega)} \lesssim 1    
      \end{equation}
      for $1<p\leq 2$,
      % \begin{equation}\label{g0H1}
      % \vert g_0\vert_{H^{1}(\Omega)} \lesssim |\ln{\underline{h}_l}|    
      % \end{equation}
    and  \begin{equation}\label{estimate_g1_1p}
        \vert g_1\vert_{W^{1,p}(\Omega)} \lesssim h_T^{\frac{2}{p}-2},
    \end{equation}
  for $\frac{4}{3}- \varepsilon_1<p\leq  2$, where $\varepsilon_1>0$ depends only on $\Omega$.
\end{lemma}
\begin{proof}
 We first prove \eqref{estimate_g1_1p}. Let $\delta_1\in C_0^{\infty}(T)$ be defined as in \eqref{prodT}. For any $x \in T $, we  denote by $l_x$ the straight line containing $x$, which is parallel to the $x_1$-axis. Let  $\hat{x}$ be the leftmost intersection point of $l_x$ and $\partial T$. We denote by $l_{\hat{x}x}$ the line segment with $\hat{x}$ and $x$ as the endpoints. Define
 \begin{equation*}
     F_1:=\nu_1\int_{l_{\hat{x}x}} \frac{\partial\delta_1}{\partial x_1} \text{d}s
 \end{equation*}
 for $x\in T$. For $x\notin T$, we let $F_1=0$.
 It may be seen that
 \begin{equation*}
     \frac{\partial F_1}{\partial x_1}=\nu_1\frac{\partial\delta_1}{\partial x_1}.
 \end{equation*}
 Using \eqref{prodT}, we have
 \begin{equation*}
     \norm{F_1}_{L^p(\Omega)}\lesssim h_T^{\frac{2}{p}-2}
 \end{equation*}
 for any $1\leq p \leq \infty$.
 Similarly, we may define $F_2$, which satisfies
 \begin{equation*}
     \frac{\partial F_2}{\partial x_2}=\nu_2\frac{\partial\delta_1}{\partial x_2}
 \end{equation*}
 and 
  \begin{equation*}
     \norm{F_2}_{L^p(\Omega)}\lesssim h_T^{\frac{2}{p}-2}
 \end{equation*}
 for any $1\leq p \leq \infty$. Let $F=(F_1,F_2)$. It is obvious  that 
 \begin{equation*}
     \text{div} F=\nabla\delta_1\cdot \nu.
 \end{equation*}
 Owing to Lemma \ref{lemma gradient estimate}, we obtain
\begin{align}\label{g1pguodu}
	\vert g_{1}\vert_{W^{1,p}(\Omega)}&\lesssim \norm{F}_{L^p(\Omega)}\lesssim h_T^{\frac{2}{p}-2}
\end{align}
for $\frac{4}{3}- \varepsilon_1<p\leq  2$, where $\varepsilon_1>0$ depends only on $\Omega$. Using a similar argument, we may prove \eqref{estimate_g0_1p}. 
\end{proof}

With the help of Lemma \ref{estg_1}, we may obtain some estimates for $R_lg_1$.
\begin{lemma}\label{lem estR_lg_1}
    It holds that
\begin{equation}\label{estimate_g1_H1}
        \vert R_lg_1\vert_{H^1(\Omega)}\lesssim h_T^{-1}
    \end{equation}
    and
    \begin{equation}\label{estimate_g1_Lp}
        \Vert R_lg_1\Vert_{L^2(\Omega)}\lesssim \overline{h}_l^{r}h_T^{-1}
    \end{equation}
    for $r<r_0$.
\end{lemma}
\begin{proof}
    In conjunction with \eqref{prodT} and \eqref{def_g1}, we have
\begin{align*}
    \vert R_lg_1\vert_{H^1(\Omega)}^2 &=a(R_lg_1,R_lg_1)=(R_lg_1,\nabla \delta_1 \cdot {\nu})\\
    &=(\nabla R_lg_1\cdot {\nu}, \delta_1)\lesssim
     h_T^{-1}\vert R_lg_1\vert_{H^1(\Omega)},
\end{align*}
which yields \eqref{estimate_g1_H1}.
By a standard dual argument, we have
 \begin{align}\label{eR_l}
  \Vert g_1-R_lg_1\Vert_{L^2(\Omega)}\lesssim \overline{h}_l^r|g_1-R_lg_1|_{H^1(\Omega)}\lesssim \overline{h}_l^rh_T^{-1}
 \end{align}
 for $r\in (0,r_0)$.
Combining \eqref{estimate_g1_1p} and \eqref{eR_l}, by the Sobolev embedding theorem, we obtain 
\begin{align*}    \norm{R_lg_1}_{L^2(\Omega)}&\lesssim
\norm{g_1}_{L^2(\Omega)}+\norm{g_1-R_lg_1}_{L^2(\Omega)}\\
&\lesssim h_T^{\frac{2}{p}-2}+\overline{h}_l^rh_T^{-1}
\end{align*}
for any $r<{r_0}$ and $\frac{4}{3}-\varepsilon_0<p<2$. Choosing $p\leq \frac{2}{r+1}$, we get \eqref{estimate_g1_Lp}.
\end{proof}

\section{A posteriori error analysis}
In this section, we propose two pointwise a posteriori error estimators, including computable and theoretical estimators. Specifically,  the computable (local) estimator is defined as
\begin{equation}\label{def eta}
    \eta_l(T):=h_T^2\sum_{i\in J}\Vert \lambda_{l,i}u_{l,i}+\Delta u_{l,i}\Vert_{L^{\infty}(T)}+h_T\sum_{i\in J}\Vert[\![\frac{\partial u_{l,i}}{\partial {n}}]\!]\Vert_{L^{\infty}(\partial T\backslash\partial \Omega)},\ \ 
\eta_l:=\max_{T\in\mathcal{T}_l} \eta_l(T),
\end{equation}
and the theoretical estimator associated with the index $j$ is defined as
\begin{equation}\label{def eta*}
\eta_{l,j}^*:=\inf_{v_l\in V_l}\max_{T\in\mathcal{T}_l}\left(h_T^2\Vert \Delta\Lambda_l u_j+v_l\Vert_{L^{\infty}(T)}+h_T\Vert[\![\frac{\partial \Lambda_l u_j}{\partial {n}}]\!]\Vert_{L^{\infty}(\partial T\backslash\partial \Omega)}\right).
\end{equation}
For $j\in J$, denote $e_j:=u_j-\Lambda_l u_j$ where $\Lambda_l$ is defined in \eqref{deflambdal}. 
 The adaptive algorithm consists of the following process:
\par \textbf{Input.} Give an initial triangulation $\mathcal{T}_0$ and the bulk parameter $0<\theta\leq 1$.
\par for $l=0,1,2,...$ do
\par \textbf{Solve.} Compute discrete solutions $(\lambda_{l,j},u_{l,j})$ for $j\in J$ of \eqref{def_u_h} with respect to $\mathcal{T}_l$ .
\par \textbf{Estimate.} Compute local contributions of the error estimator $\eta_l(T)$ for all $T\in \mathcal{T}_l$.
\par \textbf{Mark.} We mark all elements satisfying 
\begin{equation*}
    \eta_l(T)\geq \theta \eta_l.
\end{equation*}
\par \textbf{Refine.} Generate a new triangulation $\mathcal{T}_{l+1}$ by using the BiSecLG(1) refinement strategy in \cite{MR4320894}.
\par end do
\par \textbf{Output.} Sequences of triangulations $\{\mathcal{T}_l\}_{l}$, discrete solutions $\{(\lambda_{l,j},u_{l,j}\}_{j\in J}\}_l$ and the a posteriori error estimators $\{\eta_l\}_l$.

\begin{remark}
   In the subsequent analysis for the reliability of $\eta_{l}$, we will use the weighted Sobolev stability of $Q_l$ in the sense of $L^1$-norm and $W^{1,1}$-norm. Therefore, we adopt the BiSecLG(1) refinement algorithm to promise the validity of Lemma \ref{weighted Sobolev stability}. From numerical experiments, the performance of above adaptive algorithm with other refinement strategies  liking red-green-blue refinement and newest vertex bisection is similar to that of the BiSecLG(1) refinement strategy.
\end{remark}

Now we give the main result in this paper, namely the reliability and efficiency of $\eta_{l}$.

\begin{theorem}\label{main result}
     Let $\Omega$ be an arbitrary polygonal domain. Assume  \eqref{eq:H2} holds. If the mesh-size $\overline{h}_l$ is small enough,  the pointwise a posteriori error estimator $\eta_l$ has the  reliability and efficiency, up to some logarithmic factors, namely
\begin{equation}\label{theorem reliablility}
          ||e_j||_{L^{\infty}(\Omega)}\lesssim \frac{|\ln{\underline{h}_l}|^2}{|\ln{\overline{h}_l}|}\frac{\lambda_{n+L}}{\lambda_{n+1}}\eta_l
\end{equation}
for any $j\in J$ and
\begin{equation}\label{theorem efficiency}
        \eta_l \lesssim \sqrt{L}\sum_{j\in J}||e_j||_{L^{\infty}(\Omega)}.
\end{equation}
\end{theorem}
\par The proof of \eqref{theorem reliablility} may be found in Corollary \ref{cor reliability} in subsection 4.2, and \eqref{theorem efficiency} may be seen in Lemma \ref{lemma Efficiency} in subsection 4.3.

\subsection{Proper approximation}
Before the a posteriori error analysis, we shall briefly explain that $\Lambda_l u_j$ is a ``proper" approximation of $u_j$ ($j\in J$) in some sense.
\begin{lemma}\label{lem proper approximation}
 Let $(\lambda_j,u_j)\in \mathbb{R}^+\times W$ $(j\in J)$ be an eigenpair of \eqref{Laplace}, then we have  
 \begin{equation}\label{proper  approximation}
  \norm{u_j-\Lambda_l u_j}_{L^{\infty}(\Omega)}\lesssim \norm{u_j-R_lu_j}_{L^{\infty}(\Omega)}.
    \end{equation}
\end{lemma}
\begin{proof}
The triangle inequality directly leads to
\begin{equation}\label{Lemma_proper1}
    \norm{u_j-\Lambda_lu_j}_{L^{\infty}(\Omega)}\leq  \norm{u_j-R_lu_j}_{L^{\infty}(\Omega)} + \norm{R_lu_j-\Lambda_lu_j}_{L^{\infty}(\Omega)}.
\end{equation}
 Let $\tilde{x}_0\in \tilde{T}_0$ satisfy $|(R_lu_j -\Lambda_lu_j)(\tilde{x}_0)|=\norm{R_lu_j -\Lambda_lu_j}_{L^{\infty}(\Omega)} $. Then there exists a regularized delta function  $\tilde{\delta}_0$ satisfying 
\begin{equation*}
	 (\tilde{\delta}_0, v)_{\tilde{T}_0}=v(\tilde{x}_0) \ \ \forall\  v\in P_k(\tilde{T}_0),\  \Vert \tilde{\delta}_0\Vert_{W^{s,\infty}(\Omega)}\lesssim h_{\tilde{T}_0}^{-s-2},s=0,1,2,....
\end{equation*}
We may define a regularized Green's function $\tilde{g}_0$ as in \eqref{defg}. With the help of Lemmas \ref{lemma u-Lambda_l u} and \ref{lemma Lambda_l}, we have 
\begin{align}\label{Lemma_proper2}\begin{aligned}
    \norm{R_lu_j -\Lambda_lu_j}_{L^{\infty}(\Omega)}&=|(R_lu_j -\Lambda_lu_j,\tilde{\delta}_0)|
    = |a(R_lu_j -\Lambda_lu_j,R_l \tilde{g}_0)|\\
    &=|a(u_j-\Lambda_lu_j,R_l\tilde{g}_0)|=\lambda_j|(u_j-P_lu_j,R_l\tilde{g}_0)|\\
    &\leq \lambda_j \norm{u_j-P_lu_j}_{L^2(\Omega)} \norm{R_l\tilde{g}_0}_{L^2(\Omega)}\\
    &\lesssim \norm{u_j-\Lambda_lu_j}_{L^2(\Omega)}\norm{R_l\tilde{g}_0}_{L^2(\Omega)}\\
    &\lesssim  \norm{u_j-R_lu_j}_{L^2(\Omega)}\norm{R_l\tilde{g}_0}_{L^2(\Omega)}.
    \end{aligned}
\end{align}
With a standard dual argument and using \eqref{estimate_g0_1p}, we may deduce
\begin{align}\label{Lemma_proper3}\begin{aligned}
    \norm{R_l \tilde{g}_0}_{L^2(\Omega)}&\leq \norm{\tilde{g}_0-R_l\tilde{g}_0}_{L^2(\Omega)}+\norm{\tilde{g}_0}_{L^2(\Omega)}\\
    &\lesssim \overline{h}_l^{r}\norm{\tilde{g}_0-R_l\tilde{g}_0}_{H^1(\Omega)}+\norm{\tilde{g}_0}_{W^{1,p}(\Omega)}\\
    &\lesssim \overline{h}_l^{r}\norm{\tilde{g}_0}_{H^1(\Omega)}+\norm{\tilde{g}_0}_{W^{1,p}(\Omega)}\\
    &\lesssim 1+\overline{h}_l^{r}\lesssim 1,
    \end{aligned}
\end{align}
where $1<p<2$ and $r<r_0$. Combining \eqref{Lemma_proper1}, \eqref{Lemma_proper2} and \eqref{Lemma_proper3}, we may complete the proof of \eqref{proper  approximation}.
\end{proof}

\begin{remark} [An a priori $L^{\infty}$-estimate]
    According to Lemma \ref{lem proper approximation}, we may obtain the a priori $L^{\infty}$-estimate for elliptic eigenvalue problems. In particular, if $\Omega$ is convex, and the mesh is qausi-uniform, we have (see \cite{i3} and references therein)
\begin{equation*}
      \norm{u_j-\Lambda_l u_j}_{L^{\infty}(\Omega)}\lesssim c_l\overline{h}_l\inf_{v_l\in V_l}\norm{u_j-v_l}_{W^{1,\infty}(\Omega)},
\end{equation*}
where $c_l=1+|\ln \overline{h}_l|$ for the linear element and $c_l=1$ for higher-order elements.
If $\Omega$ is non-convex and the mesh is local refined as in \cite{MR4435940}, we may obtain
\begin{equation*}
      \norm{u_j-\Lambda_l u_j}_{L^{\infty}(\Omega)}\lesssim c_l\inf_{v_l\in V_l}\norm{u_j-v_l}_{L^{\infty}(\Omega)}.
\end{equation*}
\end{remark}

\begin{remark}
We shall emphasize that the Ritz projection $R_l$ generally loses the stability in the sense of $L^{\infty}$-norm in non-convex domain, although $R_lu_j$ is the best approximation of $u_j$ in the sense of the energy norm. It is difficult to find a (qausi-) optimal and computable approximation in the sense of $L^{\infty}$-norm in highly graded meshes. Thus, with assuming that $R_lu_j$ is a ``good" approximation of $u_j$, we deem that it is reasonable to regard $\Lambda_l u_j$ as  a proper approximation of $u_j$ in view of \eqref{proper  approximation}.
\end{remark}

\subsection{Reliability}
We first review the weighted Sobolev stability of the $L^2$-projection $Q_l$ in \cite{MR4320894}, which provides a new method different from \cite{MR3532806} to prove the reliability of $\eta_l$ when $k\geq 2$. Let the mesh-size function $h_l$ have the grading $\gamma_l$ (see Definition 4.1 in \cite{MR4320894}). We have the following 
weighted Sobolev stability of the $L^2$-projection $Q_l$.
\begin{lemma}\label{weighted Sobolev stability}
    Assume that 
    \begin{equation}\label{grading}
        \gamma_l^{2+|1-\frac{2}{p}|}< \gamma_{max}:=\frac{\sqrt{2k+2}+\sqrt{k}}{\sqrt{2k+2}-\sqrt{k}}.
    \end{equation}
Then the $L^2$-projection $Q_l$ has the weighted Sobolev stability 
\begin{equation}\label{weightedLp}   \norm{h_l^{-2}Q_lv}_{L^p(\Omega)}\lesssim \norm{h_l^{-2}v}_{L^p(\Omega)} \ \ \ \ \forall\ v\in L^{p}(\Omega)
\end{equation}
and
 \begin{equation}\label{weightedW1p}   \norm{h_l^{-1}\nabla Q_lv}_{L^p(\Omega)}\lesssim \norm{h_l^{-1}\nabla v}_{L^p(\Omega)} \ \ \ \ \forall\ v\in W_0^{1,p}(\Omega)
\end{equation}
for $1\leq p\leq \infty$.
\end{lemma}

\par For $j\in J$, let $x_0\in \overline{\Omega}$ satisfy 
\begin{equation}\label{defK0}
    \vert e_j(x_0)\vert=\Vert e_j\Vert_{L^{\infty}(\Omega)}.
\end{equation}
For the subsequent analysis, our first step is to prove 
\begin{equation}\label{assosiation111}
    \Vert e_j\Vert_{L^{\infty}(\Omega)}\lesssim \vert (\delta_0,e_j)\vert,
\end{equation}
where $\delta_0$ is defined in \eqref{prodelta1} and \eqref{prodelta2}. 
We may use a similar argument as in \cite{1} to get the following lemma. For simplicity, we omit the detailed proof. 
\begin{lemma}\label{lem nochetto}
Let $0< \alpha\leq 1$ be the H{\"o}lder continuity exponent of the eigenfunction $u_j$. If the mesh-size  $\overline{h}_l$ is small enough, the following estimate holds
\begin{equation}\label{associate}
    \Vert e_j\Vert_{L^{\infty}(\Omega)}\lesssim
    \vert (\delta_0,e_j)\vert + h_{T_0}^{\alpha\beta},
\end{equation}
where $T_0$ is an element containing $x_0$ and $\beta\geq 1$.
\end{lemma}

\par In order to drop the second term in \eqref{associate}, we use the  a priori lower bound of $\Vert e_j\Vert_{L^{\infty}(\Omega)}$. Thanks to Lemmas \ref{lemma lower bound} and \ref{lem nochetto}, choosing $\beta > \frac{2}{\alpha}\frac{|\ln{\underline{h}_l}|}{|\ln \overline{h}_l|}$, we may deduce \eqref{assosiation111}. Note that 
\begin{equation}\label{g0W21}
    \norm{g_0}_{W^{2,1}(\Omega)}\lesssim \frac{|\ln{\underline{h}_l}|^2}{|\ln{\overline{h}_l}|}
\end{equation}
under the choice of $\beta$ and in view of \eqref{smooth2}. 
\par Next we shall derive the upper bound of $\Vert e_j\Vert_{L^{\infty}(\Omega)}$. 
Elementary algebraic manipulations lead to
\begin{align}\label{align_dual_delta0_ej}
    (\delta_0,e_j)=a(e_j,g_0)&=a(u_j,g_0)-a(\Lambda_l u_j,Q_l g_0)+a(\Lambda_l u_j,Q_l g_0-g_0).
\end{align}
With the aid of Lemmas \ref{lemma u-Lambda_l u}, \ref{lemma Lambda_l} and \ref{weighted Sobolev stability}, using the standard interpolation theory, we may obtain
\begin{align}\label{A1}
\begin{aligned}
       a(e_j,g_0)&=\lambda_j \left( (u_j,g_0)-(P_l u_j,Q_l g_0)\right)+a(\Lambda_l u_j,Q_l g_0-g_0) \\ 
    &=\lambda_j (u_j-P_l u_j,g_0)+a(\Lambda_l u_j,Q_l g_0-g_0)\\
  &\lesssim \lambda_j \Vert u_j-P_l u_j\Vert_{L^{2}(\Omega)}\norm{g_0}_{L^2(\Omega)}\\
  &\ \ \ \ + \eta_{l,j}^*\left(\norm{h_l^{-2}(g_0-Q_lg_0)}_{L^1(\Omega)}+\norm{h_l^{-1}(g_0-Q_lg_0)}_{W^{1,1}(\Omega)}\right)\\
    &\lesssim  \Vert u_j-R_l u_j\Vert_{L^{2}(\Omega)}\norm{g_0}_{L^2(\Omega)}+ |g_0|_{W^{2,1}(\Omega)}\eta_{l,j}^*\\
    &\lesssim \Vert u_j-R_l u_j\Vert_{L^{2}(\Omega)}+ \frac{|\ln{\underline{h}_l}|^2}{|\ln{\overline{h}_l}|}\eta_{l,j}^*.
\end{aligned}
\end{align}
In some a posteriori error analysis in the sense of energy norm (see \cite{MR3347459}),  $\norm{u_j-R_lu_j}_{L^2(\Omega)}$ may be controlled by the a priori estimate
\begin{equation*}
    \norm{u_j-R_lu_j}_{L^2(\Omega)}\lesssim \overline{h}_l^r\norm{u_j-R_lu_j}_{H^1(\Omega)},
\end{equation*}
where $r>0$ is some regularity parameter. However, in the error analysis in the sense of $L^\infty$-norm, it seems difficult to bound the $L^2$-norm error through some a priori estimates. Instead, we may use some a posteriori estimates to bound this term. 

Using the similar argument as in \cite{1} and combining the weighted Sobolev stability of $L^2$-projection, we may get the following a posteriori error estimate.  
\begin{lemma}\label{posteriori source}
    Let $u$ solve the source problem 
  \begin{align}\label{Laplace source}
	\begin{cases}
		-\Delta u=f \ \ &\text{in} \ \Omega,\\
		u=0 \ \ &\text{on} \  \partial \Omega.
	\end{cases}
\end{align}  
Then we have 
    \begin{equation*}
     \norm{u-R_lu}_{L^{p}(\Omega)}\lesssim  \hat{\eta}_{l,p}
    \end{equation*}
for $4<p<\infty$, where $\hat{\eta}_{l,p}$ is defined as
    \begin{equation*}
  \hat{\eta}_{l,p}:=   \inf_{v_l\in V_l}\left(\sum_{T\in \mathcal{T}_l}\left(\norm{h_T^2(f+\Delta R_lu+v_l)}_{L^p(T)}^p+\norm{h_T^{1+\frac{1}{p}}[\![\frac{\partial R_l u}{\partial n}]\!]}_{L^{p}(\partial T\backslash \partial \Omega)}^p\right)\right)^{\frac{1}{p}}.    
    \end{equation*}
\end{lemma}
\begin{proof}
For $4<p<\infty$, let $\omega$ solve
  \begin{align*}
	\begin{cases}
		-\Delta \omega=\text{sgn}(u-R_lu)|u-R_lu|^{p-1} \ \ &\text{in} \ \Omega,\\
		\omega=0 \ \ &\text{on} \  \partial \Omega.
	\end{cases}
\end{align*} 
By a standard dual argument, integrating by parts and using the interpolation theory, the trace theorem and \eqref{priori}, we deduce
\begin{align*}
    \norm{u-R_lu}_{L^p(\Omega)}^p &=(u-R_lu,\text{sgn}(u-R_lu)|u-R_lu|^{p-1})\\
    &=a(u-R_lu,\omega-Q_l\omega)\\
    &=\sum_{T\in \mathcal{T}_l} \left( (u-R_lu,\omega-Q_l\omega)_T+ \langle [\![\frac{\partial R_lu}{\partial n}]\!],\omega-Q_l\omega\rangle_{\partial T\backslash \partial \Omega}\right)\\
    &\lesssim |\omega|_{W^{2,q}(\Omega)}\hat{\eta}_{l,p}\lesssim \norm{u-R_lu}_{L^p(\Omega)}^{p-1}\hat{\eta}_{l,p},
\end{align*}
where $\frac{1}{p}+\frac{1}{q}=1$. Thus, we have
\begin{equation*}
     \norm{u-R_lu}_{L^{p}(\Omega)}\lesssim  \hat{\eta}_{l,p}
\end{equation*}
for any $4<p<\infty$.
\end{proof}

Taking $u=u_j$, $f=\lambda_ju_j$ and $v_{l}=-\lambda_{j}R_{l}u_{j}+w_{l}$ for any $w_{l}\in V_{l}$, a direct triangle inequality leads to 
\begin{align*}
            \norm{u_j-R_lu_j}_{L^{p}(\Omega)}^p&\lesssim  
    \inf_{w_l\in V_l}\sum_{T\in \mathcal{T}_l}\left(\norm{h_T^2(\Delta R_lu_j+w_l)}_{L^p(T)}^p+\norm{h_T^{1+\frac{1}{p}}[\![\frac{\partial R_l u_j}{\partial n}]\!]}_{L^{p}(\partial T\backslash \partial \Omega)}^p\right)    \\
 &\ \ \ \ + \norm{h_l^2\lambda_j(u_j-R_lu_j)}_{L^{p}(\Omega)}^p
\end{align*}
for $4<p<\infty$. For small enough $\overline{h}_l$, we may infer that
\begin{equation}\label{Corollary1}
           \norm{u_j-R_lu_j}_{L^{p}(\Omega)}\lesssim  \hat{\eta}_{l,j},
\end{equation}
where $\hat{\eta}_{l,j}$ is defined as 
\begin{equation*}
  \hat{\eta}_{l,j}:=   \inf_{w_l\in V_l}\max_{T\in \mathcal{T}_l}\left(h_T^2\norm{\Delta R_lu_j+w_l}_{L^{\infty}(T)}^p+h_T\norm{[\![\frac{\partial R_l u_j}{\partial n}]\!]}_{L^{\infty}(\partial T\backslash \partial \Omega)}\right).    
\end{equation*}
Notice that \eqref{Corollary1} is valid for all $1<p<\infty$. Combining \eqref{assosiation111}, \eqref{align_dual_delta0_ej}, \eqref{A1} and \eqref{Corollary1}, we may conclude
\begin{align}\label{upper bound}
\norm{e_j}_{L^{\infty}(\Omega)}&\lesssim \max_{T\in\mathcal{T}_l}\left(h_T^2\norm{\Delta (R_lu_j-\Lambda_l u_j)}_{L^\infty(T)}+h_T\norm{[\![\frac{\partial}{\partial n }(R_lu_j-\Lambda_l u_j)]\!]}_{L^{\infty}(\partial T\backslash \partial \Omega)}\right)\notag\\
&\ \ \ \ +\frac{|\ln{\underline{h}_l}|^2}{|\ln{\overline{h}_l}|}\eta_{l,j}^*.
\end{align}
Now we show that the first term in \eqref{upper bound} is a higher order quantity.

\begin{lemma}\label{lem Rlu -Lambdal u}
For any $T\in \mathcal{T}_l$, it holds that
   \begin{equation}\label{R_lu -Lambda_l u}
  \norm{\nabla (R_l u_j-\Lambda_l u_j)}_{L^{\infty}(T)}\lesssim h_T^{-1}\overline{h}_l^r\norm{e_j}_{L^{\infty}(\Omega)}
\end{equation}
and 
\begin{equation}\label{R_lu -Lambda_l u 2}
    \norm{\Delta (R_lu_j-\Lambda_l u_j)}_{L^{\infty}(T)}\lesssim h_T^{-2}\overline{h}_l^r\norm{e_j}_{L^{\infty}(\Omega)}
\end{equation}
for any $r<r_0$.
\end{lemma}
\begin{proof}
For any $T\in \mathcal{T}_l$, let $x_T\in T$ satisfy
\begin{align*}
  \norm{\nabla (R_l u_j-\Lambda_l u_j)}_{L^{\infty}(T)}\lesssim \sup_{\nu} |(\nabla(R_l u_j-\Lambda_l u_j)\cdot {\nu}
 )(x_T)|,
\end{align*}
where $\nu$ is an arbitrary direction vector. Let $\delta_1$ be defined as in \eqref{prodT} associated with $x_T$.
Using Lemmas \ref{lemma Lambda_l} and \ref{lem estR_lg_1},  we deduce 
\begin{align*}\label{difference}
\begin{aligned}
	(\nabla(R_l u_j-\Lambda_l u_j)\cdot {\nu}
 )(x_T)&=(R_l u_j-\Lambda_l u_j,\nabla \delta_1 \cdot {\nu})\\
 &=a(R_l u_j-\Lambda_l u_j,R_l g_1)
	=\lambda_j(u_j-P_l u_j,R_l g_1)\\
    &\leq \Vert u_j-P_l u_j\Vert_{L^{2}(\Omega)}\Vert R_l g_1\Vert_{L^2(\Omega)}\\
    &\leq \Vert u_j-\Lambda_l u_j\Vert_{L^{2}(\Omega)}\Vert R_l g_1\Vert_{L^2(\Omega)}\\
    &\lesssim h_T^{-1}\overline{h}_l^{r}\norm{e_j}_{L^{\infty}(\Omega)}
\end{aligned}
\end{align*}
for any $r<r_0$,
which leads to \eqref{R_lu -Lambda_l u}. Then we may get \eqref{R_lu -Lambda_l u 2} through a standard inverse estimate. 
\end{proof}

\begin{corollary}[Reliability]\label{cor reliability}
Assume  \eqref{eq:H2} holds. If the mesh-size $\overline{h}_l$ is small enough, it holds that
\begin{equation*}
    \norm{e_j}_{L^{\infty}(\Omega)}\lesssim \frac{|\ln{\underline{h}_l}|^2}{|\ln{\overline{h}_l}|}\frac{\lambda_{n+L}}{\lambda_{n+1}}\eta_l
\end{equation*}
for any $j\in J$.
\end{corollary}
\begin{proof}
   Combining \eqref{upper bound} and Lemma \ref{lem Rlu -Lambdal u}, we get
    \begin{equation*}
         \norm{e_j}_{L^{\infty}(\Omega)}\lesssim \frac{|\ln{\underline{h}_l}|^2}{|\ln{\overline{h}_l}|}\eta_{l,j}^*. 
    \end{equation*}
We expand $\Lambda_lu_j=\sum_{i\in J} \alpha_iu_{l,i}$ with coefficients $\alpha_i=(R_lu_k,u_{l,i})=\lambda_{l,i}^{-1}\lambda_j(u_j,u_{l,i})\leq \frac{\lambda_{n+L}}{\lambda_{n+1}}$. Taking $v_l=\sum_{i\in J}\alpha_i \lambda_{l,i}u_{l,i}$ in \eqref{def eta*} and using the triangle inequality, we complete the proof.
\end{proof}

\begin{remark}
Note that 
$\Delta \Lambda_l u_j =0$ when the linear element is used. Taking $v_l=0$ in $\eta_l^*$, with a similar argument, we may obtain that 
 \begin{equation*}
         \norm{e_j}_{L^{\infty}(\Omega)}\lesssim \frac{|\ln{\underline{h}_l}|^2}{|\ln{\overline{h}_l}|}\frac{\lambda_{n+L}}{\lambda_{n+1}}\left(  \max_{T\in \mathcal{T}_l} h_T\sum_{i\in J}\Vert[\![\frac{\partial u_{l,i}}{\partial {n}}]\!]\Vert_{L^{\infty}(\partial T\backslash\partial \Omega)}\right).
    \end{equation*}
    In other words, the jump flux residuals dominate the a posterori error when $k=1$.
\end{remark}

\subsection{Efficiency}
In this subsection, we give the proof of efficiency of $\eta_l$.
\begin{lemma}[Efficiency]\label{lemma Efficiency}
 If the mesh-size $\overline{h}_l$ is small enough,   it holds that
    \begin{equation*}
       \eta_l\lesssim \sqrt{L}\sum_{j\in J}\norm{e_j}_{L^{\infty}(\Omega)}.
    \end{equation*}
\end{lemma}
\begin{proof}
We modify the local argument in \cite{1,i20} and use some results in \cite{MR3532806}.  Let $b_T$ be the usual bubble function, i.e., the polynomial of degree three, obtained by the product of the barycentric coordinates of $T$. Note that (see \cite{i20})
\begin{equation*}
     \Vert b_T\Vert_{W^{s,1}(T)}\lesssim h_T^{2-s}.
\end{equation*}
Integrating by parts, for any $v\in V$, we have
    \begin{equation*}
        a(e_j,v)=(e_j,-\Delta v)+\sum_{T\in \mathcal{T}_l}\langle e_j,\frac{\partial v}{\partial n}\rangle_{\partial T}
    \end{equation*}
and
\begin{equation*}
     a(e_j,v)=(\lambda_j u_j+\Delta \Lambda_lu_j,v)+\sum_{T\in \mathcal{T}_l}\langle v,\frac{\partial e_j}{\partial n}\rangle_{\partial T}.
\end{equation*}
Taking $v=(\lambda_j P_lu_j+\Delta\Lambda_lu_j)b_T$ and using the trace theorem, we get
\begin{align}\label{efficiency 1}
\begin{aligned}
   h_T^2\norm{\lambda_j P_lu_j+\Delta\Lambda_lu_j}_{L^{\infty}(T)}&\lesssim h_T\norm{u_j-P_lu_j}_{L^2(T)}+\norm{e_j}_{L^{\infty}(T)}\\
    &\lesssim h_T\norm{u_j-\Lambda_lu_j}_{L^2(\Omega)}+\norm{e_j}_{L^{\infty}(T)}\\
    &\lesssim \norm{e_j}_{L^{\infty}(\Omega)}.    
\end{aligned}
\end{align}
 Given an interior edge $E$, we denote by $T_1$ and $T_2$ the two elements that share the edge $E$. Let $\varphi_E$ be a piecewise quadratic polynomial on $T_1\cup T_2$, which is one at the midpoint of $E$ and vanishes on the other edges of $T_1\cup T_2$. Taking $v=[\![\frac{\partial \Lambda_l u_j}{\partial {n}}]\!] \varphi_E$, with a similar argument as in proving  \eqref{efficiency 1}, we obtain
 \begin{equation} \label{efficiency 2}
     h_{E}\norm{[\![\frac{\partial \Lambda_l u_j}{\partial {n}}]\!]}_{L^{\infty}(E)}\lesssim \norm{e_j}_{L^{\infty}(\Omega)}.
 \end{equation}
 Owing to  Lemma 1 in \cite{MR3532806}, we have
\begin{equation*}
\sum_{j\in J}\norm{\lambda_{l,j}u_{l,j}+\Delta u_{l,j}}_{L^2(T)}^2\lesssim    \sum_{j\in J}\norm{\lambda_jP_lu_j+\Delta\Lambda_lu_j}_{L^2(T)}^2\ 
\end{equation*}
for any $T\in \mathcal{T}_l$, and 
\begin{equation*}
\sum_{j\in J}\norm{[\![\frac{\partial u_{l,j}}{\partial {n}}]\!]}_{L^{2}(E)}^2  \lesssim  \sum_{j\in J}\norm{[\![\frac{\partial \Lambda_l u_j}{\partial {n}}]\!]}_{L^{2}(E)}^2
\end{equation*}
for any interior edge $E$. Using the H{\"o}lder inequality and the stand inverse inequality, and combining \eqref{efficiency 1} and \eqref{efficiency 2}, we deduce 
\begin{equation*}
    {\eta_l^2(T)}\lesssim L\sum_{j\in J} \norm{e_j}_{L^{\infty}(\Omega)}^2\lesssim L(\sum_{j\in J}\norm{e_j}_{L^{\infty}(\Omega)})^2,
\end{equation*}
which gives the efficiency of $\eta_l$ in view of $\eta_l=\max_{T\in \mathcal{T}_l}\eta_l(T)$.
\end{proof}

 \section{Numerical experiments}
In this section, we present some benchmarks of numerical experiments in non-convex domains where some eigenfunctions have singularities. Specifically, we approximate \eqref{Laplace} in different domains with the linear conforming element by the adaptive algorithm proposed in section 4.1. We set the bulk parameter $\theta=0.5$ in all adaptive computations. 
% All convergence graphs are logarithmically scaled, the error quantities are plotted against the degrees of freedom.

\subsection{Approximation on L-shaped domain}
We start our numerical experiments with the approximation in L-shaped domain 
\begin{equation*}
    \Omega_1=(0,1)^2 \backslash ([\frac{1}{2},1)\times (0,\frac{1}{2}]).
\end{equation*}
The initial mesh $\mathcal{T}_0$ is an uniform mesh with $h_T=|T|^{\frac{1}{2}}=\frac{\sqrt{2}}{16}$ for all elements. By the adaptive algorithm, it may be seen that $\lambda_{12}$ is close to $\lambda_{13}$ and the eigenfunction $u_{12}$ (associated with $\lambda_{12}$) has local singularities. Thus, in this experiment, we take $J=\{12,13\}$. Figure \ref{fig:L1mesh} shows the adaptive mesh with 19700 elements, from which we can see that our
adaptive algorithm captures the singularities especially around the reentrant corner accurately. Figure \ref{fig:L1rate} shows that the convergence rate of $\eta_l$ is $O(N^{-1})$, which is quasi-optimal. Moreover, from Table \ref{tab:L} we may see that the product
of $N_l$ and $\eta_l$ is gradually stable around a certain constant, which illustrates that $\eta_l$ is $O(N^{-1})$.

\begin{figure}[H]
\centering
\begin{minipage}[t]{0.48\textwidth}
\centering
\includegraphics[width=6cm,height=4cm]{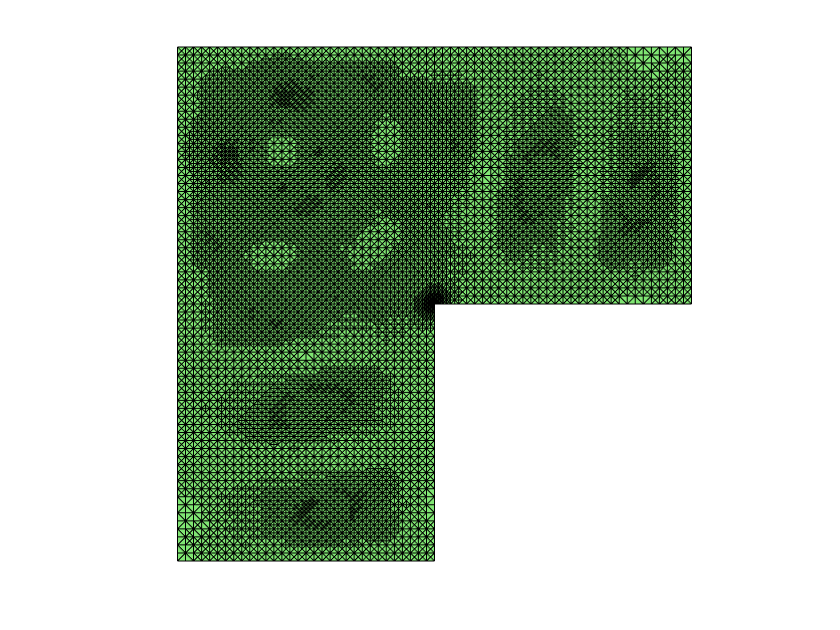}
\caption{The adaptive mesh with $19700$ elements in $\Omega_{1}$ for the 12th and 13th eigenpairs.}\label{fig:L1mesh}
\end{minipage}
\hspace{0.15in}
\begin{minipage}[t]{0.48\textwidth}
\centering
\includegraphics[width=6cm,height=4cm]{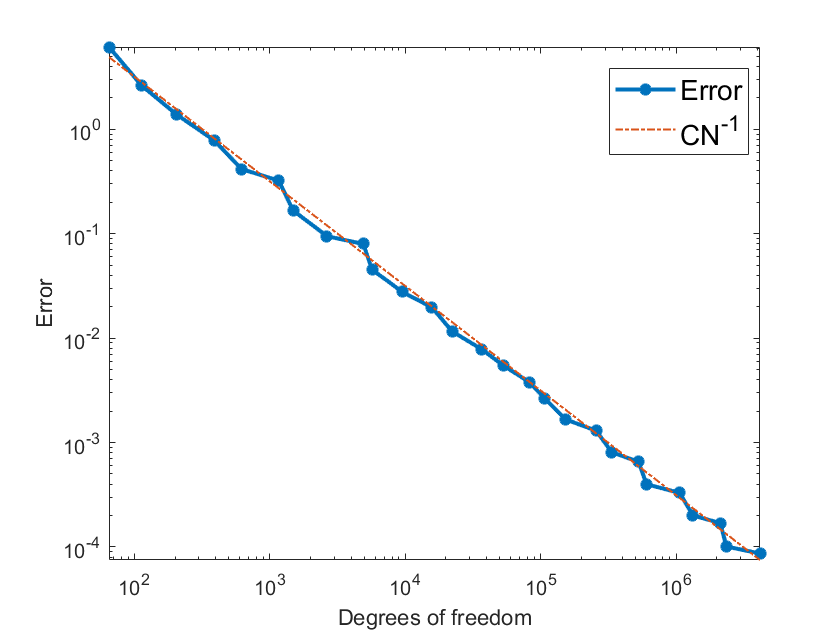}
\caption{The blue dots represent the pointwise a posteriori error estimator $\eta_l$ for $J=\{12,13\}$ in $\Omega_1$. The dotted line represents the fitted convergence rate.}\label{fig:L1rate}
\end{minipage}
\end{figure}

\begin{table}[H]
    \centering    
    \begin{tabular}{cccc}\hline
       \textbf{mesh level}  & \textbf{degrees of freedom}&\textbf{a posteriori error}&\textbf{product} \\
        $l$& $N_l$& $\eta_l$&$ N_l\times \eta_l$\\
        \hline

0 & 33 & 37.768&   1246.3\\
4 & 469 &  2.2696&     1064.4\\
8& 2625 & 0.47195 & 1251.6\\
12&  7508 &  0.17591 &   1320.7\\
16& 18981 &    0.068936&    1308.5\\
% 20 &43034   &0.0272 &1155.7\\
 \hline
    \end{tabular}
    \caption{Degrees of freedom, the a posteriori error and their product in L-shape domain $\Omega_1$ for 12th and 13th eigenpairs on different mesh levels.}
    \label{tab:L}
\end{table}

\par Next, we present some numerical results to illustrate that when we limit the a prior error in the sense of $L^{\infty}$-norm in practical applications, it is better to use the $L^{\infty}$ error estimator. To show this, we introduce the $H^{1}$ error estimator defined in \cite{MR3347459}
\begin{equation*}
   \eta_l^{a}(T)=\left(\sum_{j\in J} h_T^2\norm{\lambda_{l,j}u_{l,j}}_{L^2(T)}^2+\sum_{j\in J}h_T\norm{[\![\frac{\partial u_{l,j}}{\partial n}]\!]}^2_{L^2(\partial T \backslash \partial\Omega)} \right)^{\frac{1}{2}} 
\end{equation*}
and
\begin{equation*}
    \eta_l^{a}=(\sum_{T\in \mathcal{T}_l}  \eta_l^{a}(T)^2)^{\frac{1}{2}}.
\end{equation*}

\begin{figure}[H]
\centering
\begin{minipage}[t]{0.48\textwidth}
\centering
\includegraphics[width=6cm,height=4cm]{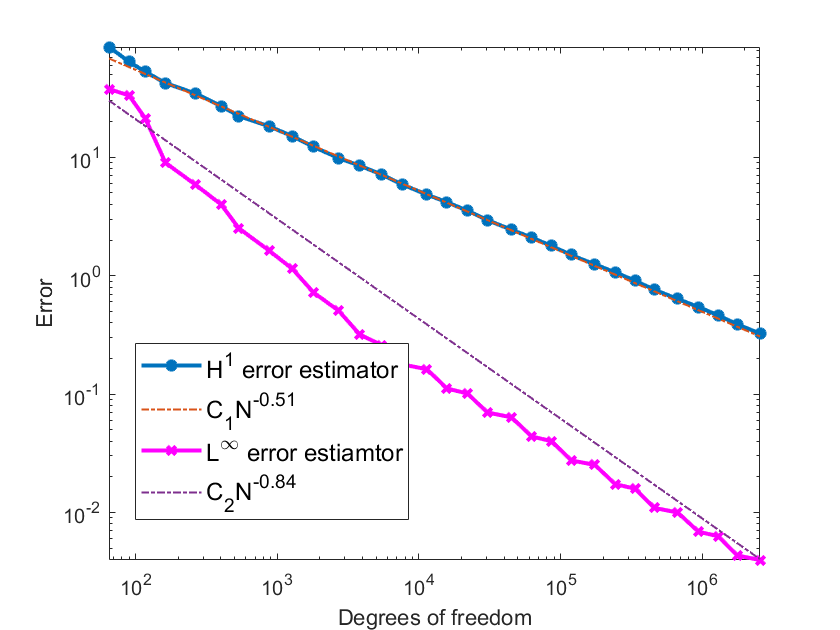}
\caption{The convergence rates of two error estimators when using the $H^{1}$ error estimator.}\label{fig:Lcom1}
\end{minipage}
\hspace{0.15in}
\begin{minipage}[t]{0.48\textwidth}
\centering
\includegraphics[width=6cm,height=4cm]{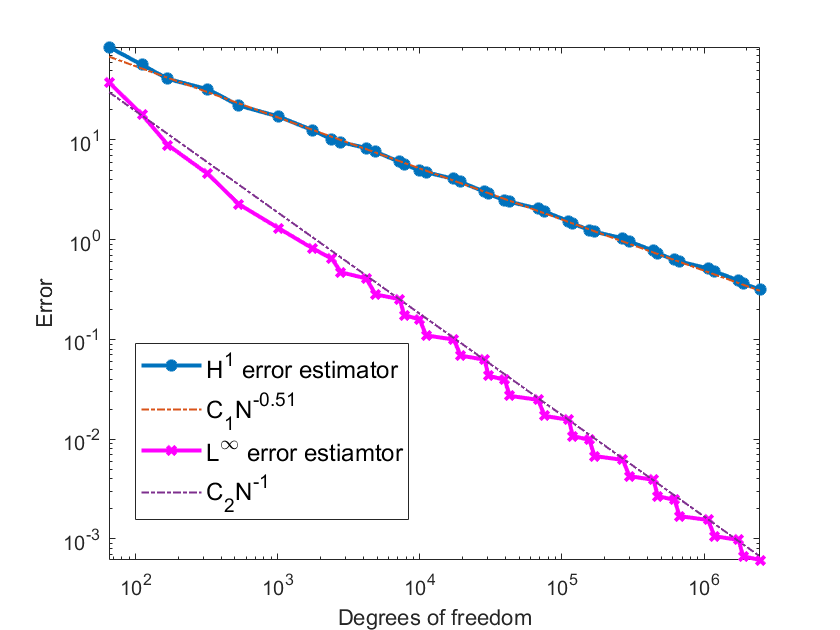}
\caption{The convergence rates of two error estimators when using the $L^{\infty}$ error estimator.}\label{fig:Lcom2}
\end{minipage}
\end{figure}

\par When we use the $H^1$ error estimator, we rely on D{\"o}rfler marking in \textbf{Mark}. First we use the $H^1$ error estimator in the adaptive algorithm. In each iteration, we also record the value of the $L^{\infty}$ error estimator (it may be seen as the a priori $L^{\infty}$ error in view of the reliability  and efficiency of the $L^{\infty}$ error estimator). Figure \ref{fig:Lcom1} shows that when the $H^1$ error estimator is used, the convergence rate of the $L^{\infty}$ error estimator becomes suboptimal. In other words, we need more degrees of freedom to limit the a priori error in the sense of $L^{\infty}$-norm. Then, we use the $L^{\infty}$ error estimator as the a posteriori error estimator in the adaptive algorithm. We also record the value of the $H^1$ error estimator in each iteration (it may be seen as the a priori $H^1$ error). Figure \ref{fig:Lcom2} shows that  the convergence rate of the $H^1$ error estimator keeps optimal. This experiment illustrates that we may get the optimal convergence rate of the a priori error in the sense of energy norm by using the $L^{\infty}$ error estimator.

\subsection{Approximation on a symmetric slit domain}
\par In this subsection, numerical experiments are set in a square domain with four symmetric slits defined as
\begin{align*}
    \Omega_2=&(-1,1)^2\backslash 
    ( \text{conv}\{(0.5, 0), (1, 0)\} \cup \text{conv}\{(0, 0.5), (0, 1)\} \\
    &\cup \text{conv}\{(-0.5, 0), (-1, 0)\} \cup \text{conv}\{(0, -0.5), (0, -1)\} ),
\end{align*}
where ${\rm conv}\{(a_1,b_1),(a_2,b_2)\}$ denotes the line segment with $(a_1,b_1)$ and $(a_2,b_2)$ as the endpoints.

\begin{figure}[H]
\centering
\begin{minipage}[t]{0.48\textwidth}
\centering
\includegraphics[width=6cm,height=4cm]{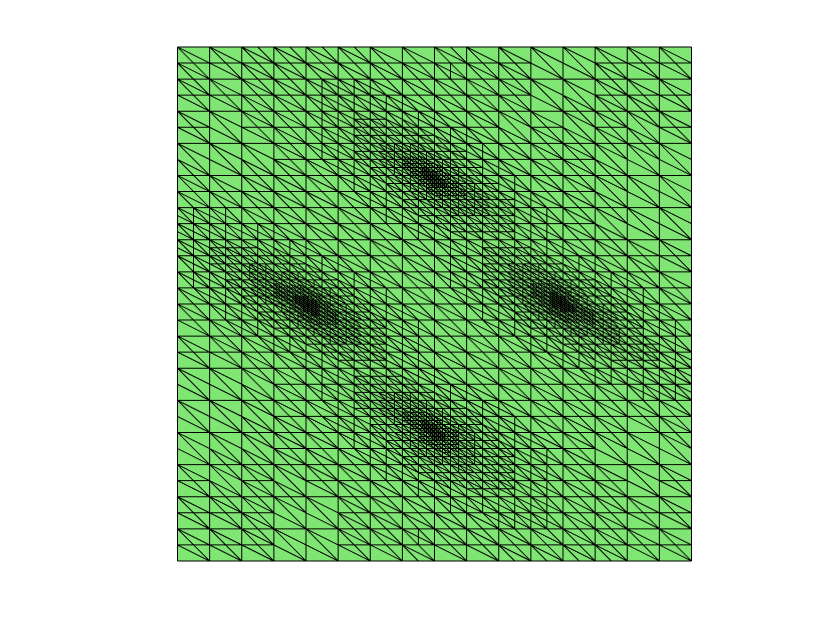}
\caption{The adaptive mesh with $5373$ elements in $\Omega_{2}$ for the 2nd and 3rd eigenpairs.}\label{fig:C1mesh}
\end{minipage}
\hspace{0.15in}
\begin{minipage}[t]{0.48\textwidth}
\centering
\includegraphics[width=6cm,height=4cm]{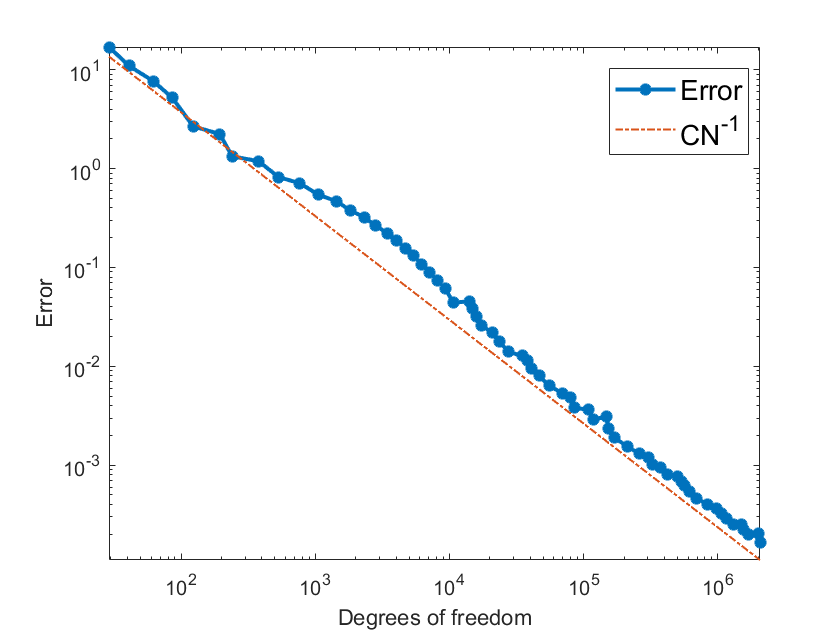}
\caption{The blue dots represent the pointwise a posteriori error estimator $\eta_l$ for $J=\{2,3\}$ in $\Omega_2$. The dotted lines represent the fitted convergence rate.}\label{fig:C1rate}
\end{minipage}
\end{figure}

\par It is known that the interior eigenvalues $\lambda_2=\lambda_3$ (see \cite{MR3347459}). 
We restrict our attention on computations for the second and third eigenpairs, i.e., $J=\{2,3\}$.
Figure \ref{fig:C1mesh} shows the adaptive mesh with 5373 elements, which illustrates that our
adaptive algorithm captures the  singularities around four slits accurately. Figure \ref{fig:C1rate} shows that the convergence rate of the pointwise a posteriori error estimator $\eta_l$ is quasi-optimal. This experiment illustrates the pointwise a poseteriori error estimator $\eta_l$ is still reliable and efficiency for multiple eigenvalues.

\subsection{Approximation on a perturbed slit domain}
In this subsection, we study a perturbation of the geometry and show that the refinement strategy may exert a significant impact on the performance of the performance of the entire adaptive algorithm. Specifically, we consider the slit domain 
\begin{align*}
    \Omega_3=&(-1,1)^2\backslash 
    ( \text{conv}\{(0.505, 0), (1, 0)\} \cup \text{conv}\{(0, 0.501), (0, 1)\} \\
    &\cup \text{conv}\{(-0.499, 0), (-1, 0)\} \cup \text{conv}\{(0, -0.5), (0, -1)\} ) .
\end{align*}

\begin{figure}[H]
\centering
\begin{minipage}[t]{0.48\textwidth}
\centering
\includegraphics[width=6cm,height=4cm]{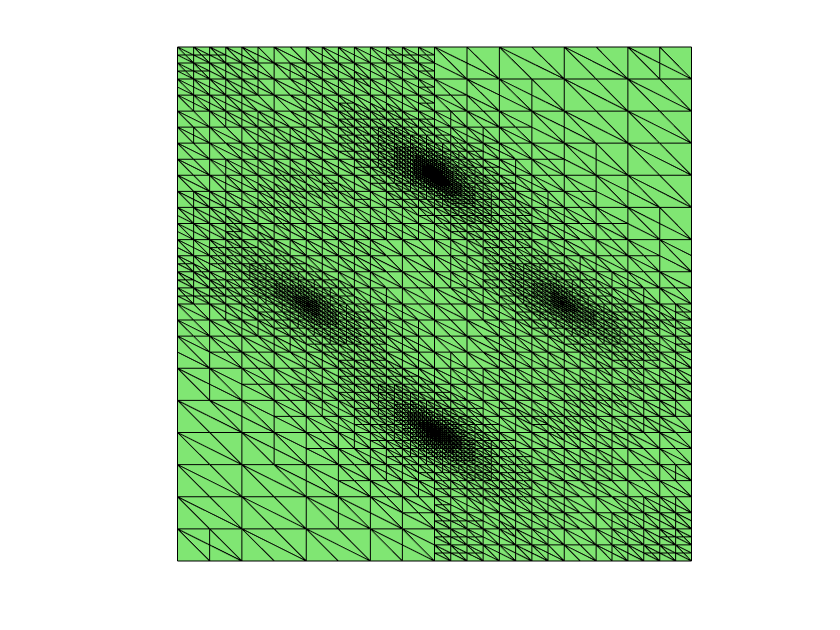}
\caption{The adaptive mesh with $5852$ elements in $\Omega_3$ for the 2nd  eigenpair by using the BiSecLG(1) refinement strategy.}\label{fig:Cwuran_Br_mesh}
\end{minipage}
\hspace{0.15in}
\begin{minipage}[t]{0.48\textwidth}
\centering
\includegraphics[width=6cm,height=4cm]{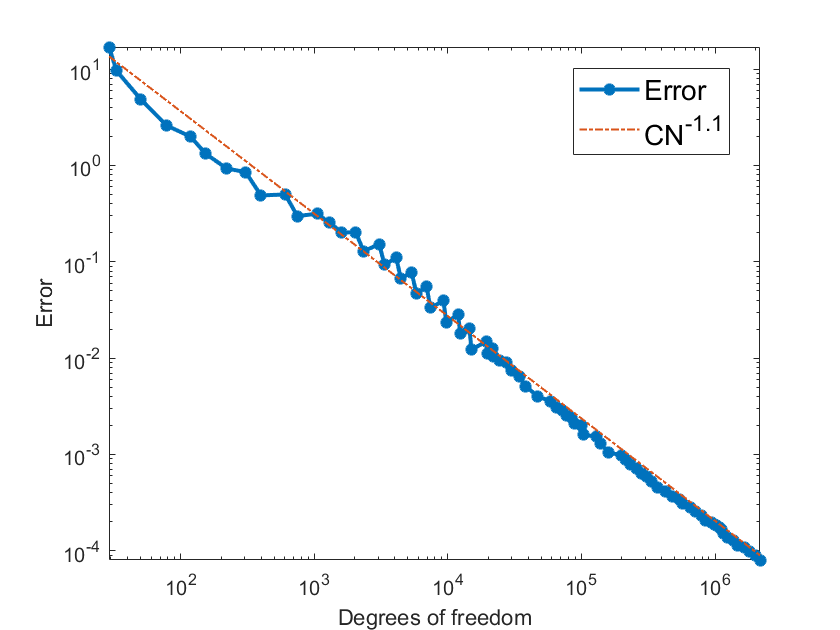}
\caption{The blue dots represent the pointwise a posteriori error estimator $\eta_l$ for $J=\{2\}$ in $\Omega_3$. The dotted lines represent the fitted convergence rate.}\label{fig:Cwuran_Br_rate}
\end{minipage}
\end{figure}

\begin{figure}[H]
\centering
\begin{minipage}[t]{0.48\textwidth}
\centering
\includegraphics[width=6cm,height=4cm]{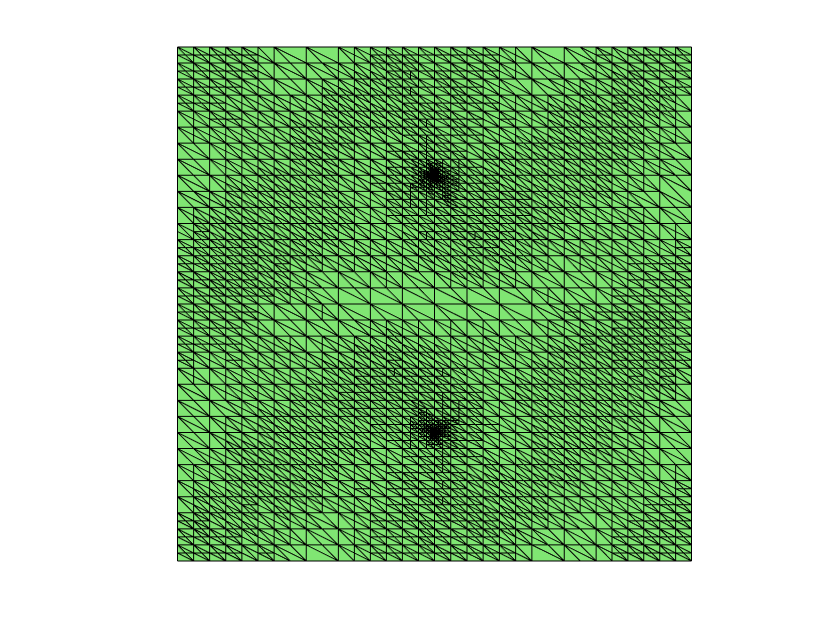}
\caption{The adaptive mesh with $5858$ elements in $\Omega_3$ for the 2nd eigenpair by using the newest vertex refinement strategy.}\label{fig:Cwuran_newest_mesh}
\end{minipage}
\hspace{0.15in}
\begin{minipage}[t]{0.48\textwidth}
\centering
\includegraphics[width=6cm,height=4cm]{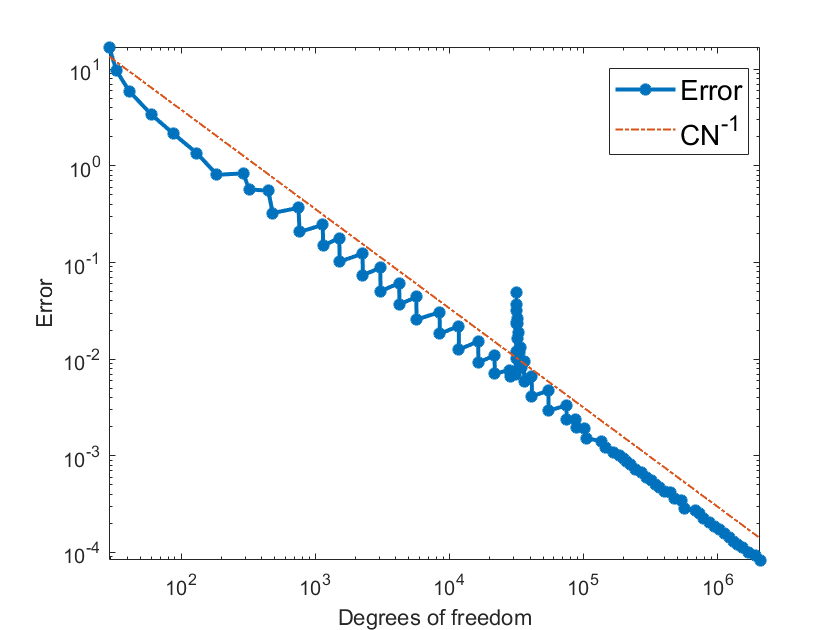}
\caption{The blue dots represent the pointwise a posteriori error estimator $\eta_l$ for $J=\{2\}$ in $\Omega_3$. The dotted lines represent the fitted convergence rate.}\label{fig:Cwuran_newest_rate}
\end{minipage}
\end{figure}

\par  
We observed an interesting phenomenon when applying the adaptive finite element method to solve clustered eigenvalue problems with different refinement strategies in Figures \ref{fig:Cwuran_Br_mesh} and \ref{fig:Cwuran_newest_mesh}. When we only choose $J=\{2\}$ and use the adaptive algorithm in section 4, we may see that  the singularities around four slits are all captured, and the convergence of $\eta_l$ is optimal (as shown in Figures \ref{fig:Cwuran_Br_mesh} and \ref{fig:Cwuran_Br_rate}). However, if we use the newest vertex refinement strategy, only two singularities are captured, and the convergence history of $\eta_l$ seems problematic when degrees of freedom within the range of $10^4$ to $10^5$ (as shown in Figures \ref{fig:Cwuran_newest_mesh} and \ref{fig:Cwuran_newest_rate}). When we choose $J=\{2,3\}$, i.e, a whole cluster of eigenvalues, this problem is cured. No matter whether we use the newest vertex refinement or BiSecLG(1) refinement, the adaptive algorithm captures the singularities around four slits (Figures \ref{fig:C22mesh} and \ref{fig:C2wuran_new_mesh}), and the convergence rate of $\eta_l$ keeps optimal as shown in Figures \ref{fig:C22rate} and \ref{C2wuran_new_rate}. This comparative experiment explains why our adaptive algorithm is more robust in the sense that it only requires some gap conditions of the cluster from the remaining spectrum, but no gap within the cluster.  And this experiment demonstrates that we should choose a whole cluster to construct  the a posteriori error estimator in practical applications.

\begin{figure}[H]
\centering
\begin{minipage}[t]{0.48\textwidth}
\centering
\includegraphics[width=6cm,height=4cm]{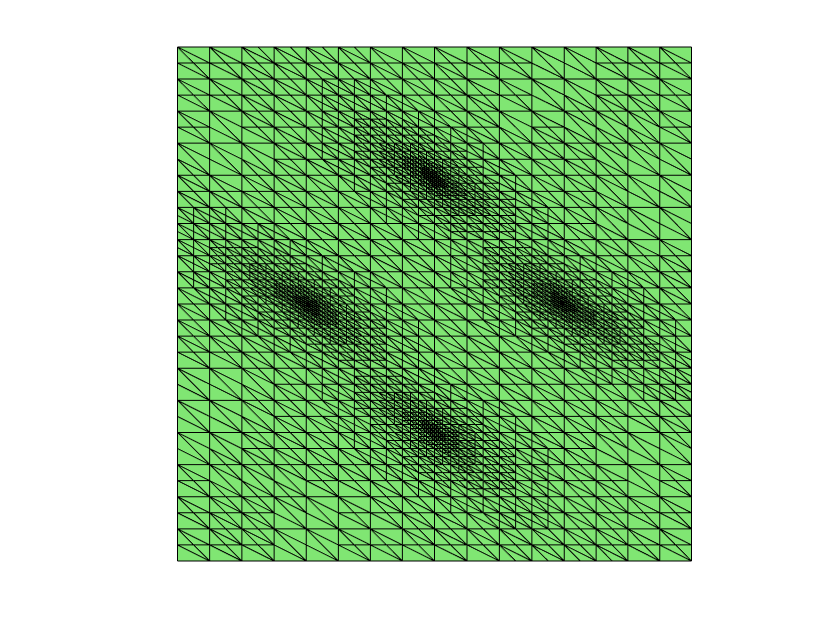}
\caption{The adaptive mesh with $5304$ elements in $\Omega_3$ for the 2nd and 3rd eigenpairs by using the BiSecLG(1) refinement strategy.}\label{fig:C22mesh}
\end{minipage}
\hspace{0.15in}
\begin{minipage}[t]{0.48\textwidth}
\centering
\includegraphics[width=6cm,height=4cm]{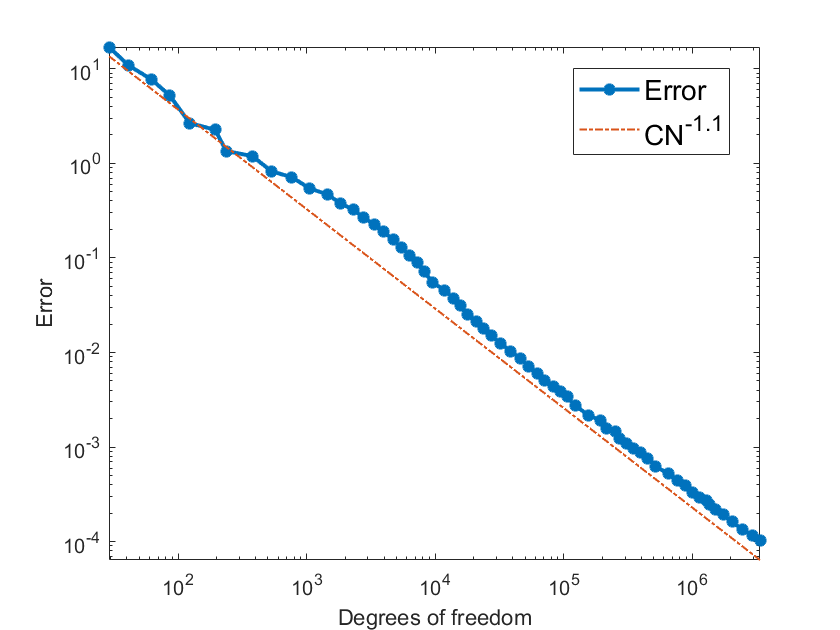}
\caption{The blue dots represent the pointwise a posteriori error estimator $\eta_l$ for $J=\{2,3\}$ in $\Omega_3$. The dotted lines represent the fitted convergence rate.}\label{fig:C22rate}
\end{minipage}
\end{figure}

\begin{figure}[H]
\centering
\begin{minipage}[t]{0.48\textwidth}
\centering
\includegraphics[width=6cm,height=4cm]{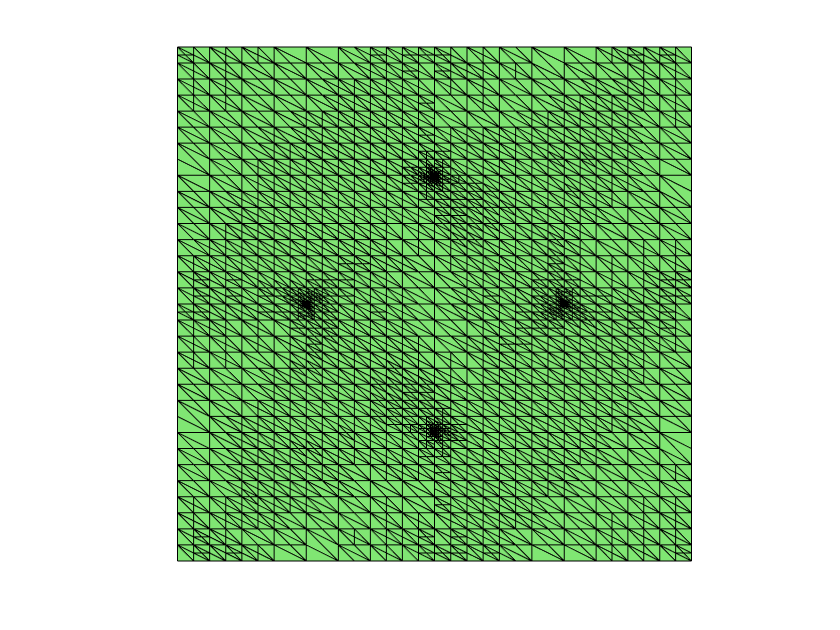}
\caption{The adaptive mesh with $4251$ elements in $\Omega_3$ for the 2nd and 3rd eigenpairs by using the newest vertex refinement  strategy.}\label{fig:C2wuran_new_mesh}
\end{minipage}
\hspace{0.15in}
\begin{minipage}[t]{0.48\textwidth}
\centering
\includegraphics[width=6cm,height=4cm]{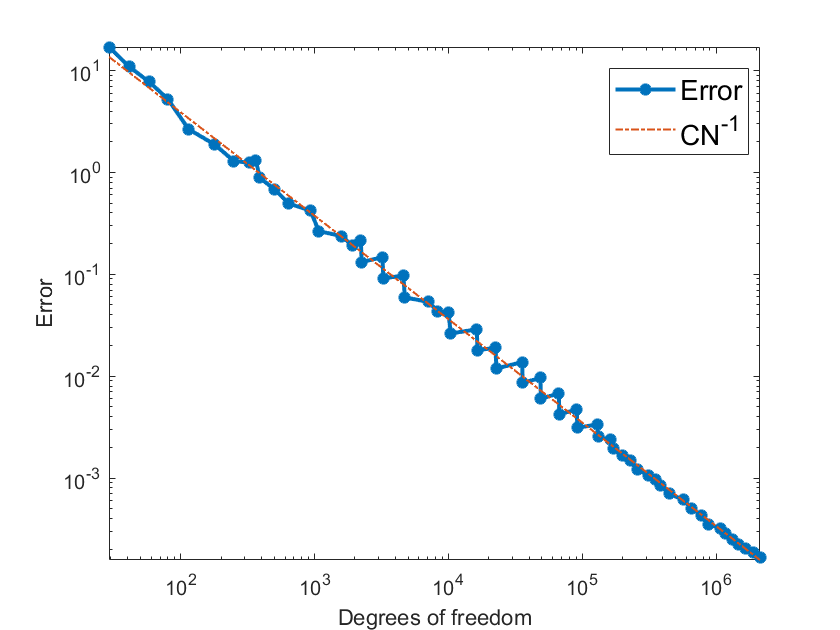}
\caption{The blue dots represent the pointwise a posteriori error estimator $\eta_l$ for $J=\{2,3\}$ in $\Omega_3$. The dotted lines represent the fitted convergence rate.}\label{C2wuran_new_rate}
\end{minipage}
\end{figure}

\section{Conclusion}
In this paper, we propose and analyze a pointwise a posteriori error estimator for conforming
finite element approximations of eigenvalue clusters of second-order elliptic eigenvalue problems with AFEMs. It is proven that the pointwise a posteriori error estimator is reliable and efficient, which is also extended to  higher order elements with the help of the weighted Sobolev stability of the $L^2$- projection. Specially, jump residuals dominate the total a posteriori error when the linear element is used. Numerical results verify our theoretical findings.

\begin{small}
\bibliographystyle{plain}
\bibliography{references}
\end{small}
\end{document}